\documentclass[11pt]{amsart}
\usepackage{amscd}
\def\NZQ{\Bbb}               

\def\ZZ{{\NZQ Z}}

\def\frk{\frak}               

\def\Phi{{\frk n}}
\def\Phi{{\frk N}}
\def\opn#1#2{\def#1{\operatorname{#2}}} 
\opn\chara{char} \opn\length{\ell} \opn\pd{pd} \opn\rk{rk}
\opn\projdim{proj\,dim} \opn\injdim{inj\,dim} \opn\rank{rank}
\opn\depth{depth} \opn\grade{grade} \opn\height{height}
\opn\embdim{emb\,dim} \opn\codim{codim}

\opn\Tr{Tr} \opn\bigrank{big\,rank}
\opn\superheight{superheight}\opn\lcm{lcm}
\opn\trdeg{tr\,deg}
\opn\reg{reg} \opn\lreg{lreg} \opn\ini{in} \opn\lpd{lpd}
\opn\size{size}
\opn\div{div} \opn\Div{Div} \opn\cl{cl} \opn\Cl{Cl}
\opn\Spec{Spec} \opn\Supp{Supp} \opn\supp{supp} \opn\Sing{Sing}
\opn\Ass{Ass} \opn\Min{Min}
\opn\Ann{Ann} \opn\Rad{Rad} \opn\Soc{Soc}
\opn\Im{Im} \opn\Ker{Ker} \opn\Coker{Coker} \opn\Am{Am}
\opn\Hom{Hom} \opn\Tor{Tor} \opn\Ext{Ext} \opn\End{End}
\opn\Aut{Aut} \opn\id{id}

\opn\nat{nat}
\opn\pff{pf}
\opn\Pf{Pf} \opn\GL{GL} \opn\SL{SL} \opn\mod{mod} \opn\ord{ord}
\opn\Gin{Gin} \opn\Hilb{Hilb}\opn\sdepth{sdepth}
\opn\aff{aff} \opn\con{conv} \opn\relint{relint} \opn\st{st}
\opn\lk{lk} \opn\cn{cn} \opn\core{core} \opn\vol{vol}
\opn\link{link} \opn\star{star}
\opn\gr{gr}

\def\pot#1#2{#1[\kern-0.28ex[#2]\kern-0.28ex]}

\opn\dirlim{\underrightarrow{\lim}}
\opn\inivlim{\underleftarrow{\lim}}
\let\Dirsum=\bigoplus

\let\to=\rightarrow

\def\Implies{\ifmmode\Longrightarrow \else
        \unskip${}\Longrightarrow{}$\ignorespaces\fi}
\def\implies{\ifmmode\Rightarrow \else
        \unskip${}\Rightarrow{}$\ignorespaces\fi}
\def\iff{\ifmmode\Longleftrightarrow \else
        \unskip${}\Longleftrightarrow{}$\ignorespaces\fi}

\let\:=\colon
\newtheorem{Theorem}{Theorem}[section]
\newtheorem{Lemma}[Theorem]{Lemma}
\newtheorem{Corollary}[Theorem]{Corollary}
\newtheorem{Proposition}[Theorem]{Proposition}
\newtheorem{Remark}[Theorem]{Remark}

\newtheorem{Example}[Theorem]{Example}

\let\epsilon\varepsilon
\let\phi=\varphi
\let\kappa=\varkappa
\def\qed{\ifhmode\textqed\fi
      \ifmmode\ifinner\quad\qedsymbol\else\dispqed\fi\fi}
\def\textqed{\unskip\nobreak\penalty50
       \hskip2em\hbox{}\nobreak\hfil\qedsymbol
       \parfillskip=0pt \finalhyphendemerits=0}
\def\dispqed{\rlap{\qquad\qedsymbol}}

\opn\dis{dis}
\def\pnt{{\raise0.5mm\hbox{\large\bf.}}}

\opn\Lex{Lex}

\begin{document}

\title{An inequality between depth and Stanley depth}

\author{Dorin Popescu}
\thanks{The  Support from  the Romanian Ministry of Education, Research,
and Inovation (PN II Program, CNCSIS  ID-1903/2008) is gratefully acknowledged.}
\address{Institute of Mathematics "Simion Stoilow", University of Bucharest, P.O. Box 1-764,
Bucharest 014700, Romania}
 \email{dorin.popescu@imar.ro}

\begin{abstract}
We show that Stanley's Conjecture holds for square free monomial
ideals in five variables, that is the Stanley depth of a square free
monomial ideal in five variables is greater or equal with its depth.
\end{abstract}

\maketitle

\section*{Introduction}
Let  $S=K[x_1,\ldots,x_n]$ be  a polynomial ring in $n$ variables
over a field $K$ and $M$  a finitely generated multigraded (i.e.
$\ZZ^n$-graded) $S$-module. Given $m\in M$  a homogeneous element in
$M$ and $Z\subseteq \{x_1,\ldots,x_n\}$, let $mK[Z]\subset M$ be the
linear $K$-subspace of all elements of the form $mf$, $f\in K[Z]$.
This subspace is called Stanley space of dimension $|Z|$, if $mK[Z]$
is a free $K[Z]$-module. A Stanley decomposition of $M$ is a
presentation of the  $K$-vector space $M$ as a finite direct sum of
Stanley spaces $\mathcal{D}:\,\,M=\Dirsum_{i=1}^rm_iK[Z_i]$. Set
$\sdepth \mathcal{D}=\min\{|Z_i|:i=1,\ldots,r\}$. The number
\[
\sdepth(M):=\max\{\sdepth({\mathcal D}):\; {\mathcal D}\; \text{is a
Stanley decomposition of}\;  M \}
\]
is called Stanley depth of $M$.  R. Stanley \cite[Conjecture
5.1]{St} gave the following conjecture.

 {\bf Stanley's Conjecture} $\sdepth(M)\geq \depth(M)$
  for all finitely generated $\ZZ^n$-graded $S$-modules $M$.

    Our Theorem \ref{p}, completely based on \cite{Po}, shows that the above  conjecture
    holds when $\dim_S\ M\leq 2$. If $n\leq 5$ Stanley's Conjecture
    holds for all cyclic $S$-modules by \cite{AP} and \cite[Theorem 4.3]{Po}.

It is the purpose of our paper to study Stanley's Conjecture on
monomial square free ideals of $S$, that is:

{\bf Weak Conjecture} Let $I\subset S$ be a monomial square free
ideal. Then $\sdepth_S\ I\geq \depth_S\ I$.

Our Theorem \ref{main} gives a kind of inductive step in proving the
above conjecture, which is settled for $n\leq 5$ in our Theorem
\ref{5}. Note that the above conjecture says in fact that
$\sdepth_S\ I\geq 1+\depth_S\ S/I$ for any monomial  square free
ideal $I$ of $S$. This remind us a question raised in \cite{Ra},
saying that $\sdepth_S\ I\geq 1+\sdepth_S\ S/I$ for any monomial
ideal $I$ of $S$. This question is harder since there exist few
known properties of Stanley depth (see \cite{HVZ}, \cite{As},
\cite{Na}, \cite{Ra}), which is not the case of the usual depth (see
\cite{BH}, \cite{Vi}). A positive answer of this question in the frame of monomial square free ideals
would state the Weak Conjecture as follows: $$\sdepth_S\ I\geq 1+\sdepth_S\ S/I\geq 1+\depth_S\ S/I=\depth_S\ I,$$
 the second inequality being a consequence of \cite[Theorem 4.3]{Po}, or of our Theorem \ref{p}.

\section{Some inequalities on depth and Stanley depth}

Let $S=K[x_1,\ldots,x_{n}]$ be a polynomial ring over a field $K$,
$I\subset S$  a monomial ideal. A. Rauf stated in \cite{Ra} the
following  two results:

\begin{Proposition}
\label{a1} $\depth_S\ S/(I,x_n)\geq \depth_S\ S/I-1.$
\end{Proposition}

\begin{Corollary} \label{a2} $\depth_S\ S/(I:v)\geq \depth_S\ S/I$
for each monomial $v\not \in I$.
\end{Corollary}

It is worth to mention that these results hold only in monomial
frame. One could think about  similar questions on Stanley depth. An
analog of the above proposition in the frame of Stanley depth is
given by \cite{Ra}. The following proposition can be seen as a
possible analog of the above corollary.

\begin{Proposition} \label{a3} $\sdepth_S\ (I:v)\geq \sdepth_S\ I$
for each monomial $v\not \in I$.
\end{Proposition}
\begin{proof} By recurrence it is enough to consider the case when
$v$ is a variable, let us say $v=x_n$. Let
$\mathcal{D}:\,\,I=\oplus_{i=1}^ru_iK[Z_i]$ be a Stanley
decomposition of $I$ such that $\sdepth\ \mathcal{D}=\sdepth_S\ I$.
We will show that
$$\mathcal{D}':\,\,(I:x_n)=(\oplus_{x_n|u_i}\ (u_i/x_n)K[Z_i])\oplus(\oplus_{u_j \not \in (x_n), x_n\in Z_j}\
 u_jK[Z_j])$$
is a Stanley decomposition of $(I:x_n)$. Indeed, if $a$ is a
monomial such that $x_na\in I$ then we have $x_na=u_iw_i$ for some
$i$ and a monomial $w_i$ of $K[Z_i]$. If $x_n\not |u_i$ then
$x_n|w_i$ and so $x_n\in Z_i$. If $x_n|u_i$ then $a=(u_i/x_n)w_i$,
which shows that
$$(I:x_n)=(\Sigma_{x_n|u_i}\ (u_i/x_n)K[Z_i])+(\Sigma_{u_j \not \in (x_n), x_n\in Z_j}\
 u_jK[Z_j]).$$
Remains to show that the above sum is direct. If $x_n |u_i$,
$u_j\not \in (x_n)$, $x_n\in Z_j$ and $u_jw_j=(u_i/x_n)w_i$ for some
monomials $w_j\in K[Z_j]$, $w_i\in K[Z_i]$ then $u_j(x_nw_j)=u_iw_i$
belongs to $u_iK[Z_i]\cap u_jK[Z_j]$, which is not possible.

Thus $\mathcal{D}'$ is a Stanley decomposition of $(I:x_n)$ with
$\sdepth\  \mathcal{D}'\geq \sdepth\ \mathcal{D}=\sdepth_S\ I$,
which ends the proof.
\end{proof}

Next we present two easy lemmas necessary in the next section:
\begin{Lemma} \label{h1} Let $I\subset J$, $I\not =J$ be some monomial ideals of
$S'=K[x_1,\ldots,x_{n-1}] $. Then
$$\sdepth_S\ JS/x_nIS\geq \min\{\sdepth_S\ JS/IS, \sdepth_{S'}\
I\}.$$
\end{Lemma}
\begin{proof} From the filtration $x_nIS\subset IS\subset JS$ we
get an isomorphism of linear $K$-spaces $JS/x_nIS\cong JS/IS\oplus
IS/x_nIS$. It follows that
$$\sdepth_S\ JS/x_nIS\geq \min\{\sdepth_S\ JS/IS, \sdepth_{S}\
IS/x_nIS\}.$$ To end note that the inclusion $I\subset IS$ induces
an isomorphism of linear $K$-spaces $I\cong IS/x_nIS$, which shows
that $\sdepth_{S'}\ I=\sdepth_S\  IS/x_nIS$.
\end{proof}

\begin{Lemma}\label{h2} Let $I\subset J$, $I\not =J$ be some monomial ideals of
$S'=K[x_1,\ldots,x_{n-1}] $ and $T=(I+x_nJ)S$. Then

\begin{enumerate}
\item{} $$\sdepth\ T\geq \min\{\sdepth_{S'}\ I,\sdepth_S\ JS\},$$
\item{} $$\sdepth\ T\geq \min\{\sdepth_{S}\ JS/IS,\sdepth_S\ IS\}.$$
\end{enumerate}
\end{Lemma}
\begin{proof}
Note that $T=I\oplus x_nJS$ as linear $K$-spaces and so (1) holds.
On the other hand the filtration $0\subset IS\subset T$ induces an
isomorphism of linear $K$-spaces $T\cong IS\oplus T/IS$ and so
$$\sdepth\ T\geq \min\{\sdepth_{S}\ T/IS,\sdepth_S\ IS\}.$$
Note that the multiplication by $x_n$ induces an isomorphism of
linear $K$-spaces $JS/IS\cong T/IS$, which shows that $\sdepth_{S}\
T/IS=\sdepth_{S}\ JS/IS$. Thus (2) holds too.
\end{proof}

An important tool in the next section is the following result, which
unifies some results from \cite{Po}.
 \begin{Theorem}\label{p} Let $U,V$ be some monomial ideals of $S$ such that
 $U\subset V$, $U\not =V$. If $\dim_S\ V/U\leq 2$ then $\sdepth_S\
 V/U\geq \depth_S\ V/U$.
 \end{Theorem}
\begin{proof}
If $V/U$ is a Cohen-Macaulay $S$-module of dimension $2$ then it is
enough to apply \cite[Theorem 3.3]{Po}. If $\dim_S\ V/U=2$ but
$\depth_S\ V/U=1$ then the result follows from \cite[Theorem
3.10]{Po}. If $\dim_S\ V/U\leq 1$ then we may apply \cite[Corollary
2.2]{Po1}.
\end{proof}

\begin{Corollary}\label{3} Let $S=K[x_1,x_2,x_3]$, $I\subset J$,
$0\not =I\not =J$ be two monomial ideals. Then $\sdepth_S\ J/I\geq
\depth_S\ J/I$.
\end{Corollary}
For the proof note that $\depth_S\ J/I\leq \dim_S\ S/I\leq 2$ and
apply Theorem \ref{p}.

\section{A hard inequality}

Let  $S'=K[x_1,\ldots,x_{n-1}]$  be  a polynomial ring in $n-1$
variables over a field $K$,  $S=S'[x_n]$ and $U,V\subset S'$,
$U\subset V$ two homogeneous ideals. We want to study the depth of
the ideal $W=(U+x_nV)S$ of $S$. Actually every monomial square free
ideal $T$ of $S$ has this form because then $(T:x_n)$ is generated
by an ideal $V\subset S'$ and $T=(U+x_nV)S$ for $U=T\cap S'$.

\begin{Lemma} \label{easy} Suppose that $U\not =V$ and $r=\depth_{S'}\
S' /U=\depth_{S'}\ S' /V$. Then $r=\depth_{S'}\ V /U$ if and only if  $r=\depth_S\
S/W$.
\end{Lemma}

\begin{proof} Set $r=\depth_{S'}\
S' /U$ and choose a  sequence $f_1,\ldots, f_r$ of homogeneous
elements of $m_{n-1}=(x_1,\ldots, x_{n-1})\subset S'$, which is
regular on $S'/U$, $S'/V$ and $V/U$ simultaneously. Set ${\bar
U}=(U,f_1,\ldots,f_r)$, ${\bar V}=(V,f_1,\ldots,f_r)$. Then
tensorizing  by $S'/(f_1,\ldots,f_r)$ the exact sequence
$$0\to V/U\to S'/U\to S'/V\to 0$$
we get the exact sequence
$$0\to V/U\otimes_{S'} S'/(f_1,\ldots,f_r)\to S'/{\bar U}\to S'/{\bar V}\to 0$$
and so ${\bar V}/{\bar U}\cong V/U\otimes_{S'} S'/(f_1,\ldots,f_r)$
has depth $0$.

Note that $f_1,\ldots, f_r$ is regular also on $S/W$ and taking
${\bar W}=W+(f_1,\ldots, f_r)S$ we get $\depth_S\ S/W=\depth_S\ S/
{\bar W} +r$. Thus passing from $U,V,W$ to ${\bar U},{\bar V},{\bar
W}$ we may reduce the problem to the case $r=0$.

If $\depth_{S'}\ V /U=0$ then there exists an element $v\in V\setminus U$ such that
$(U:v)=m_{n-1}$. Thus the non-zero element of $ S/W$ induced by $v$
is annihilated by $m_{n-1}$ and $x_n$ because $v\in V$. Hence
$\depth_S\ S/W=0$.

If $\depth_{S'}\ V /U>0$ there exists a homogeneous regular element $a$  for $V/U$ in the maximal ideal of $S'$ of degree $1$ (we may reduce to the case when $K$ is infinite) . We show that $x_n+a$ is regular for $S/W$.
Let $w=\Sigma_{i=0}^s x_n^iv_i$ for some  elements $v_i$ of $S'$ such that  $(x_n+a)w\in W$. It follows that $av_0\in U$, $(v_0+av_1)\in V,\ldots,
(v_{s-1}+av_s)\in V$, $v_s\in V$ and so $v_i\in V$ for all $i$. Then $v_0\in U$ because $a$ is regular on $V/U$, that is $w\in W$.
\end{proof}

\begin{Example} {\em Let $n=4$, $V=(x_1,x_2)$, $U=V\cap (x_1,x_3)$
be ideals of $S'=K[x_1,x_2,x_3]$ and $W=(U+x_4V)S$. Then
$\{x_3-x_2\}$ is a maximal regular sequence on $V/U$ and on $S/W$ as
well. Thus $\depth_{S'}\ V/U= \depth_{S'}\ S'/U=\depth_{S'}\
S'/V=\depth_{S}\ S/W=1$.}
\end{Example}

\begin{Lemma}\label{bad} Let $I, J\subset S'$, $I\subset J$,
$I\not=J$ be two monomial ideals, $T=(I+x_nJ)S$ such that
\begin{enumerate}
\item{} $\depth_{S'}\ S'/I=\depth_S\ S/T -1,$
\item{} $\sdepth_{S'}\ I\geq 1+\depth_{S'}\ S'/I,$
\item{} $\sdepth_{S'}\ J/I\geq \depth_{S'}\ J/I.$
\end{enumerate}
Then $\sdepth_{S}\ T\geq 1+\depth_{S}\ S/T.$
\end{Lemma}

\begin{proof}
By Lemma  \ref{h2} we have $$\sdepth_S T\geq
1+\min\{\sdepth_{S'}I,\sdepth_{S'}J/I\}\geq
1+\min\{1+\depth_{S'}S'/I,\depth_{S'}J/I\}$$ using (3), (2) and
\cite[Lemma 3.6]{HVZ}. Note that in the following exact sequence
$$0\to S/JS=S/(T:x_n)\xrightarrow{x_n} S/T\to S/(T,x_n)\cong S'/I\to
0 $$ we have $\depth_S\ S/JS=\depth_{S'}\ S'/I +1$ because of (1)
and the Depth Lemma \cite[Lemma 1.3.9]{Vi}. Thus $\depth_{S'}\ S'/I
=\depth_{S'}\ S'/J$. As $\depth_{S'}\ S'/I\not =\depth_S\ S/T$ we
get $\depth_{S'}\ S'/I\not =\depth_{S'}\ J/I$ by Lemma \ref{easy}.
But $\depth_{S'}\ J/I\geq \depth_{S'}\ S'/I$ because of the Depth
Lemma  applied to the  following exact sequence
$$0\to J/I\to S'/I\to S'/J\to 0.$$
It follows that $\depth_{S'}\ J/I\geq 1+\depth_{S'}\ S'/I$ and so
$$\sdepth_{S}\ T\geq 2+\depth_{S'}\ S'/I=1+\depth_{S}\ S/T.$$
\end{proof}

\begin{Remark} {\em The above lemma introduces the difficult
hypothesis (3) and one can hope that it is not necessary at least
for square free monomial ideals. It seems this is not the case as
shows somehow the next example.}
\end{Remark}

\begin{Example} {\em Let $n=4$, $J=(x_1x_3,x_2)$, $I=(x_1x_2,x_1x_3)$
be ideals of $S'=K[x_1,x_2,x_3]$ and $T=(I+x_4J)S=(x_1,x_2)\cap
(x_2,x_3)\cap (x_1,x_4).$ Then $\{x_4-x_2,x_3-x_1\}$ is a maximal
regular sequence on  $S/T$. Thus $\depth_{S}\ S/T=2$, $ \depth_{S'}\
S'/I=\depth_{S'}\ S'/J=1$.}
\end{Example}

\begin{Lemma}\label{good}
Let  $I, J\subset S'$, $I\subset J$, $I\not=J$ be two monomial
ideals, $T=(I+x_nJ)S$ such that
\begin{enumerate}
\item{} $\depth_{S'}\ S'/I\not =\depth_S\ S/T -1,$
\item{} $\sdepth_{S'}\ I\geq 1+\depth_{S'}\ S'/I,$ \ \ \
$\sdepth_{S'}\ J\geq 1+\depth_{S'}\ S'/J.$
\end{enumerate}
Then $\sdepth_{S}\ T\geq 1+\depth_{S}\ S/T.$
\end{Lemma}

\begin{proof}
By Lemma  \ref{h2} we have $$\sdepth_S T\geq
\min\{\sdepth_{S'}I,1+\sdepth_{S'}J\}\geq 1+\min\{\depth_{S'}
S'/I,1+\depth_{S'}S'/J\}$$ using  (2). Applying Proposition \ref{a1}
we get $\depth_{S'} S'/I=\depth_{S}S/(T,x_n)\geq \depth_S S/T-1$,
the inequality being strict because of (1). We have the following
exact sequence
$$0\to S/JS=S/(T:x_n)\xrightarrow{x_n} S/T\to S/(T,x_n)\cong S'/I\to
0 .$$ If $\depth_{S'} S'/I>\depth_S S/T$ then $\depth_{S}S/JS=
\depth_S S/T$ by Depth Lemma and so
$$\sdepth_S T\geq 1+\min\{\depth_{S'} S'/I,\depth_{S}S/JS\}=1+\depth_S S/T.$$
If $\depth_{S'} S'/I=\depth_S S/T$ then $\depth_{S}S/JS\geq
\depth_{S'} S'/I$ again by Depth Lemma and thus
$$\sdepth_S T\geq 1+\depth_{S'} S'/I=1+\depth_S S/T.$$
\end{proof}

\begin{Theorem} \label{main} Suppose that the Stanley's conjecture
holds for factors $V/U$ of monomial square free ideals, $U,V\subset
S'=K[x_1,\ldots,x_{n-1}]$, $U\subset V$, that is $\sdepth_{S'}\
V/U\geq \depth_{S'}\ V/U$. Then the Weak Conjecture  holds for
monomial square free ideals of $S=K[x_1,\ldots,x_{n}].$
\end{Theorem}

\begin{proof} Let $r\leq n$ be a positive integer and $T\subset
S_r=K[x_1,\ldots,x_{r}]$ a monomial square free ideal. By induction
on $r$ we show that $\sdepth_{S_r}\ T\geq 1+\depth_{S_r}\ S_r/T$,
the case $r=1$ being trivial. Clearly, $(T:x_r)$ is generated by a
monomial square free ideal $J\subset S_{r-1}$ containing $I=T\cap
S_{r-1}$. By induction hypothesis we have  $\sdepth_{S_{r-1}}\ I\geq
1+\depth_{S_{r-1}}\ S_{r-1}/I,$ $\sdepth_{S_{r-1}}\ J\geq
1+\depth_{S_{r-1}}\ S_{r-1}/J.$ If $I=J$ then $T=IS$, $x_r$ is
regular on $S_r/T$ and we have
$$\sdepth_{S_r}\ T=1+\sdepth_{S_{r-1}}\ I\geq 2+\depth_{S_{r-1}}\ S_{r-1}/I
=1+\depth_{S_{r}}\ S_{r}/T,$$ using \cite[Lemma 3.6]{HVZ}. Now
suppose that $I\not =J$. If $\depth_{S_{r-1}}\ S_{r-1}/I\not
=\depth_{S_r}\ S_r/T -1,$ then it is enough to apply Lemma
\ref{good}. If $\depth_{S_{r-1}}\ S_{r-1}/I =\depth_{S_r}\ S_r/T
-1,$ then apply Lemma \ref{bad}.
\end{proof}

\begin{Corollary} \label{wc4} The Weak Conjecture holds in
$S=K[x_1,\ldots,x_4]$.
\end{Corollary}
\begin{proof} It is enough to apply Lemmas \ref{bad}, \ref{good}
after we show that for monomial square free ideals $I,J\subset
S'=K[x_1,\ldots ,x_3]$, $I\subset J$, $I\not =J$, $T=(I+x_4J)S$ with
$\depth_{S'}\ S'/I=\depth_S\ S/T -1,$ we have  $\sdepth_{S'}\
J/I\geq \depth_{S'}\ J/I$. But then $I\not =0$ because otherwise
$\depth_S\ S/T \leq 3=\depth_{S'}\ S'/I$, which is false. Thus
$\dim_{S'}\ J/I\leq 2$ and we may apply Corollary \ref{3}.
\end{proof}

\begin{Lemma} \label{41} Let   $I, J\subset S=K[x_1,\ldots,x_4]$,
$I\subset J$, $0\not =I\not=J$ be two monomial square free ideals
such that all the prime ideals of $\Ass_SJ/I$ have dimension $3$.
 Then    $\sdepth_{S}\
J/I\geq \depth_{S}\ J/I.$
\end{Lemma}
\begin{proof} We have $I=J\cap U$, where $U=\cap_{Q\in \Ass J/I}\  Q$. By hypothesis each such $Q$ has height $1$ and is generated by a variable. Thus $U$ is principal,
let us say $U=(f)$ for some square free monomial $f$ of $S$. Then $J/I\cong (J+(f))/(f)$ and changing $J$ by $J+(f)$ we may suppose $I=(f)$ and $\dim S/J<3$.
We show that $\depth_S S/J\leq\depth_S S/(J+(x_i))$ for some $i$. If $\depth_S S/J=2$ then $S/J$ is a Cohen-Macaulay ring of dimension $2$, take a prime $p$
of $\Ass_S S/J$, let us say $p=(x_1,x_2)$. Then
 $$J+(x_1)=\cap_{q\in \Ass_S S/J,\ x_1\in q} \ q\cap (x_1,x_3,x_4)$$
 if $(x_3,x_4)\in \Ass_S S/J$, otherwise $J+(x_1)=
\cap_{q\in \Ass_S S/J,\ x_1\in q} \ q$. Indeed if $q\in \Ass_S S/J$ contains $x_2$ then $q+(x_1)\supset p$ and can be removed from the intersection. If $(x_3,x_4)\in \Ass_S S/J$ then necessary $\Ass_S S/J$ contains a prime $(x_1,x_j)$, or $(x_2,x_j)$ for some $j=3,4$ because otherwise $S/J$ is not Cohen-Macaulay. In the first case we may remove $(x_1,x_3,x_4)$ from the intersection, in the second case we may consider $J+(x_2)$. Thus renumbering the variables we may suppose that $J+(x_1)$ is an intersection of ideals of the form $(x_1,x_j)$ for some $j>1$ and clearly $\depth_S S/J=\depth_S S/(J+(x_1))=2.$
If $\depth_S S/J=1$ and  $\depth_S S/((x_1)+J)=0$ then we must have $J=(x_2,\ldots, x_4)$ and so $(x_2)+J=J$.

From the exact sequences:
$$0\to J/(f)\to S/(f)\to S/J\to 0$$
$$0\to (J+(x_i))/(x_i)\to S/(x_i)\to S/(J+(x_i))\to 0$$
we get $$\depth_S (J+(x_i))/(x_i)=1+\depth_S S/(J+(x_i))\geq 1+\depth_S S/J=\depth_S J/(f).$$
Apply induction on $d=\deg\ f$. If $d=1$ then $f=x_i$ and we may apply Theorem \ref{main} for the ideal $(J+(x_i))/(x_i)\subset S'=S/(x_i)$. Suppose $d>1$. We have the following exact sequence
$$0\rightarrow (J\cap (x_i))/(f)\to J/(f)\to J/(J\cap (x_i))\cong (J+(x_i))/(x_i)\to 0.$$
  But $(J\cap (x_i))/(f)\cong (J:x_i)/(f')$, where $f'=f/x_i$. As $\deg f'=d-1$ we may apply the induction hypothesis to get
  $$\sdepth_S (J\cap (x_i))/(f)\geq \depth_S
  (J\cap (x_i))/(f)\geq \depth_S\ J/(f),$$
  as Depth Lemma gives from the above exact sequence. Thus $$\sdepth_S\ J/(f)\geq \min\{\sdepth_S\ (J\cap (x_i))/(f), \sdepth_S\ (J+(x_i))/(x_i)\}\geq $$
  $$\min \{\depth_S\ J/(f),\depth_S\ (J+(x_i))/(x_i)\geq \depth_S\ J/(f)$$
  by \cite{Ra} and Theorem \ref{main}.
\end{proof}

\begin{Proposition} \label{4} Let   $I, J\subset S=K[x_1,\ldots,x_4]$,
$I\subset J$, $0\not =I\not=J$ be two monomial square free ideals.
 Then    $\sdepth_{S}\
J/I\geq \depth_{S}\ J/I.$
\end{Proposition}
\begin{proof} If $\depth J/I=1$ then Stanley's Conjecture holds by \cite{C1}, \cite{C2}. If $\depth J/I=3$ we may apply Lemma \ref{41}. Suppose that $\depth_S
J/I=2$. Let $J_2/I$, $I\subset J_2\subset J$ be the largest submodule of $J/I$ of dimension $\leq 2$ (see Schenzel's dimension filtration \cite{Sc}). We have
$\Ass_S\ J/J_2=\{Q\in \Ass_S\ J/I: \dim Q=3 \}$ and $ \sdepth_S\ J/J_2\geq \depth_S\ J/J_2$ by Lemma \ref{41}. As
$\sdepth_S\ J/I\geq \min\{\sdepth_S\ J_2/I,\sdepth_S\ J/J_2\}$ by \cite{Ra} we get
$\sdepth_S\ J/I\geq \min\{\depth_S\ J_2/I,\depth_S\ J/J_2\}$
applying Theorem \ref{p} and it is enough to see that the last minimum is $\geq 2$.

Now note that $J$ is not the maximal ideal, otherwise $\depth_S\ J/I<2$. Thus $\depth_S\ S/J> 0$. As in the proof of Lemma \ref{41} we may suppose $J_2=J\cap (f)$
for some square free monomial $f$ of $S$. Thus from the exact sequence
$$0\to J/J_2\to S/(f)\to S/J\to 0$$
we get $\depth J/J_2\geq 2$ using Depth Lemma. The same argument says that $\depth J_2/I\geq 2$
using the following exact sequence
$$0\to J_2/I\to J/I\to J/J_2\to 0.$$
\end{proof}

\begin{Theorem} \label{5} The Weak Conjecture  holds in
$S=K[x_1,\ldots,x_5]$.
\end{Theorem}
For the proof note that Proposition \ref{4} gives what is necessary in the
proof of Theorem \ref{main} to pass from $S_4$ to $S_5$.


\begin{thebibliography}{99}

\bibitem{AP} I.\ Anwar and D.\ Popescu, Stanley conjecture in small embedding dimension,\ J.\ Alg.\ {\bf 318}(2007),\ 1027-1031.
\bibitem{BH} W.\ Bruns and J.\ Herzog, Cohen Macaulay Rings, Revised edition,\ Cambridge,\ 1996.
\bibitem{C1} M. Cimpoeas, {\em Stanley depth of monomial ideals in three variables,} Arxiv:math. AC/0807.2166, Preprint(2008).
\bibitem{C2} M. Cimpoeas, {\em Some remarks on the Stanley depth for multigraded modules}, Le Mathematiche, vol. LXIII (2008), Fasc. II, 165-171.


\bibitem{HVZ} J.\ Herzog, M. Vladoiu and X. Zheng, How to compute the Stanley depth of a monomial ideal, to appear in J.\ Alg.

\bibitem{Na} S.\ Nasir, Stanley decompositions and localization, Bull.\ Math.\ Soc.\ Sc.\ Math.\ Roumanie {\bf 51}(99),\ no.2 (2008),\ 151-158.
\bibitem{Po1} D.\ Popescu, Criterions for shellable multicomplexes,
An. St. Univ. Ovidius, Constanta, {\bf 14}(2), (2006), 73-84,
Arxiv:Math. AC/0505655.
\bibitem{Po} D.\ Popescu, Stanley depth of multigraded modules,\ J.\ Algebra {\bf 321}(2009), 2782-2797.
\bibitem{As} A.\ Rauf, Stanley Decompositions, Pretty Clean Filtrations and Reductions Modulo Regular Elements, Bull.\ Math.\ Soc.\ Sc.\ Math.\ Roumanie {\bf 50}(98), no.4 (2007), 347-354.
\bibitem{Ra} A.\ Rauf, Depth and Stanley depth of multigraded
modules, Arxiv:Math.\ AC/0812.2080, to appear in Communications in
Alg.


\bibitem{Sc} P.\ Schenzel, On the dimension filtrations and Cohen-Macaulay
filtered modules, In: {\em Commutative algebra and algebraic
geometry (Ferrara)}, Lecture Notes in Pure and Appl. Math. {\bf
206}, Dekker, New York, 1999, 245-264.


\bibitem{St} R.\ P.\ Stanley, Linear Diophantine Equations  and Local Cohomology, Invent.\ Math.\ {\bf 68} (1982), 175-193.
\bibitem{Vi} R.\ H.\ Villareal, Monomial Algebras, Marcel Dekker Inc.\, New York,\ 2001.
\end{thebibliography}
\end{document}